\documentclass[12pt]{article}
\usepackage{amsmath,amsfonts,amssymb, amsthm,enumerate,graphicx,xcolor,url}
\usepackage{mathtools}
\usepackage{tikz,authblk}
\usepackage{subfig}
\DeclarePairedDelimiter{\ceil}{\lceil}{\rceil}

\newtheorem{theorem} {{\textsf{Theorem}}}
\newtheorem{proposition}[theorem]{{\textsf{Proposition}}}
\newtheorem{corollary}[theorem]{{\textsf{Corollary}}}

\newtheorem{ques}[theorem]{{\textsf{Question}}}

\newtheorem{remark}[theorem]{{\textsf{Remark}}}

\newcommand{\Star}{\mbox{\upshape st}\,}
\newcommand{\lk}{\mbox{\upshape lk}\,}

\begin{document}
\title{Average edge order of normal $3$-pseudomanifolds}
\author{Biplab Basak$^1$ and Raju Kumar Gupta}
	
\date{}
	
\maketitle
	
\vspace{-12mm}
	
	\begin{center}
		
		\noindent {\small Department of Mathematics, Indian Institute of Technology Delhi, New Delhi 110016, India$^{2}$.}
		
	\end{center}
	\footnotetext[1]{{\em Corresponding author} }
	\footnotetext[2]{{\em E-mail addresses:} \url{biplab@iitd.ac.in} (B.
		Basak), \url{Raju.Kumar.Gupta@maths.iitd.ac.in} (R. K. Gupta).}

	\medskip
	
	\begin{center}
		\date{May 9, 2025}
	\end{center}
	
	\hrule

\begin{abstract}
In their work \cite{LuoStong}, Feng Luo and Richard Stong introduced the concept of the average edge order, denoted as $\mu_0(K)$. They demonstrated that if $\mu_0(K)\leq \frac{9}{2}$ for a closed $3$-manifold $K$, then $K$ must be a sphere. Building upon this foundation, Makoto Tamura extended similar results to $3$-manifolds with non-empty boundaries in \cite{Tamura1, Tamura2}. In our present study, we extend these findings to normal $3$-pseudomanifolds. Specifically, we establish that for a normal $3$-pseudomanifold $K$ with singularities, $\mu_0(K)\geq\frac{30}{7}$. Moreover, equality holds if and only if $K$ is a one-vertex suspension of a triangulation of  $\mathbb{RP}^2$ with seven vertices. Furthermore, we establish that when $\frac{30}{7}\leq\mu_0(K)\leq\frac{9}{2}$, the $3$-pseudomanifold $K$ can be derived from some boundary complexes of $4$-simplices by a sequence of possible operations, including connected sums, bistellar $1$-moves, edge contractions, edge expansions, vertex folding, and edge folding. 

\end{abstract}
\noindent {\small {\em MSC 2020\,:} Primary 05E45; Secondary 05C07, 05A20, 52B10, 52B70
	
	\noindent {\em Keywords:} Normal pseudomanifolds; Average edge order; Suspension.}

\medskip

\section{Introduction}
Consider a normal $3$-pseudomanifold denoted as $K$. Its {\em average edge order}, denoted by $\mu_0(K)$, is defined as $\mu_0(K)=\frac{3F(K)}{E(K)}$, where $F(K)$ and $E(K)$ represent the number of faces and edges in $K$ respectively. Essentially, $\mu_0(K)$ signifies the average number of faces incident on edges within $K$. In their research \cite{LuoStong}, Feng Luo and Richard Stong demonstrated that a small average edge order in a closed connected $3$-manifold suggests relatively simple topology and imposes constraints on its triangulation. 
The implication of this finding can be summarized as:

\begin{proposition}[\cite{LuoStong}]\label{LuoStong}
Let $K$ be any triangulation of a closed connected $3$-manifold $M$.
Then
\begin{enumerate}[$(i)$]
\item $3 \leq \mu_0(K) < 6$, equality holds if and only if $K$ is the triangulation of the
boundary of a $4$-simplex.
\item For any $e > 0$, there are triangulations $K_1$ and $K_2$ of $M$ such that
$\mu_0(K_1) < 4.5 + e$ and $\mu_0(K_2) > 6 - e$.
\item If $\mu_0(K) < 4.5$, then $K$ is a triangulation of $\mathbb{S}^3$.
\item If $\mu_0(K) = 4.5$, then $K$ is a triangulation of $\mathbb{S}^3$, $\mathbb{S}^2\times \mathbb{S}^1$, or   $\mathbb{S}^2\tilde{\times} \mathbb{S}^1$. Furthermore, in the last two cases the triangulations can be described.
\end{enumerate}

\end{proposition}

Building upon the previous discussion, Makoto Tamura contributed further insights in \cite{Tamura1} by extending similar results to $3$-manifolds with non-empty boundaries. In \cite{Tamura2}, Tamura revised the definition of the average edge order as follows: $\mu_0(K)=3\frac{F_0(K)}{E_0(K)}$, where  $E_0(K)= E_i (K) + \frac{E_{\partial}(K)}{2}$ and $F_0(K)= F_i (K) + \frac{F_{\partial}(K)}{2}$. Here, $E_i(K)$ (respectively, $F_i(K)$) represents the number of edges (respectively, faces) within int $K=K\setminus\partial(K)$, while $E_{\partial}(K)$ (respectively, $F_{\partial}(K)$) denotes the number of edges (respectively, faces) within $\partial(K)$. With this refined definition of the average edge order, Tamura established the following theorem:

\begin{proposition}[\cite{Tamura2}, Theorem 1.2]	
Let $K$ be any triangulation of a compact connected $3$-manifold $M$ with non-empty boundary. Then
\begin{enumerate}[$(i)$]
\item $2 \leq\mu_0(K)< 6$, and equality holds if and only if $K$ is the triangulation of
one $3$-simplex.
\item For any rational number $r$ with $4 < r < 6$, there is a triangulation $K'$ of $M$
such that $\mu_0(K') = r$.
\item If $\mu_0(K) < 4$, then $K$ is a triangulation of $\mathbb{B}^3$. There are an infinite number of distinct such triangulations, but for any constant $c < 4$ there are only finitely many triangulations $K$ with $\mu_0(K) \leq c$.
\item If $\mu_0(K) = 4$, then $K$ is a triangulation of $\mathbb{B}^3, \mathbb{D}^2\times \mathbb{S}^1$ or $\mathbb{D}^2\tilde{\times} \mathbb{S}^1$. Furthermore, in the last two cases the triangulations can be described.
\end{enumerate}

\end{proposition}

\medskip

\noindent  In \cite{Casali1, Casali2}, similar work is done for colored triangulations of manifolds. Some classification results for closed connected smooth low–dimensional manifolds according to average edge order is given in \cite{CavicchioliSpaggiari}. In this article, we establish comparable conclusions for normal $3$- pseudomanifolds without boundaries. In fact, we prove the following result:

\begin{theorem}\label{theorem:maintheorem}
Let $K$ be a normal $3$-pseudomanifold with singularities, and let $\mu_0(K)$ represent the average edge order of $K$. Then

\begin{enumerate}[$(i)$]
\item $\mu_0(K)\geq \frac{30}{7}$, and equality holds if and only if $K$ is a one-vertex suspension of $\mathbb{RP}^2$ with seven vertices.

\item If $\mu_0(K)\leq\frac{9}{2}$, then $K$ is obtained from some boundary complexes of $4$-simplices by a sequence of possible operations, including connected sums, bistellar $1$-moves, edge contractions, edge expansions, vertex folding, and edge folding. In particular, either $|K|$ is a handlebody with its boundary coned off or a suspension of $\mathbb{RP}^2$.

\item If $K$ contains exactly $n$ singular vertices, then $\mu_0(K)<6+n$. Additionally, a sequence of normal $3$-pseudomanifolds, denoted as $\{K_m\}_{m\geq 1}$, with two singularities exists, such that $\mu_0(K_m)\rightarrow 8$ as $m\rightarrow\infty$.
\end{enumerate}
\end{theorem}

\section{Preliminaries}
A {\em $d$-simplex} is the convex hull of $d+1$ affinely independent points. A $0$-simplex is a point, a $1$-simplex is an edge, a $2$-simplex is a triangle, and so on. Let $\sigma$ be a simplex. Then, {\em faces} of $\sigma$ are defined as the convex hulls of non-empty subsets of the vertex set of $\sigma$. We denote a face $\tau$ of $\sigma$ as $\tau\leq\sigma$. A {\em simplicial complex} $K$ is a finite collection of simplices in $\mathbb{R}^m$ for some $m\in\mathbb{N}$ that satisfies the following conditions: if $\sigma\in K$, then $\tau\leq\sigma$ implies $\tau\in K$; for any two faces $\sigma$ and $\tau$ in $K$, either $\sigma\cap\tau=\emptyset$ or $\sigma\cap\tau\leq\sigma$ and $\sigma\cap\tau\leq\tau$.
We assume that the {\em empty simplex} $\emptyset$ (considered as a simplex of dimension $-1$) is a member of every simplicial complex. If $K$ is a simplicial complex, then its {\em geometric carrier} $|K|=\cup_{\sigma \in K} \sigma$ is the union of all the simplices in $K$ together with the subspace topology induced from $\mathbb{R}^n$ for some $n\in \mathbb{N}$. The {\em dimension} of $K$ is defined as the maximum of the dimension of simplices in $K$. If $\sigma$ and $\tau$ are two skew simplices in $\mathbb{R}^n$ for some $n\in\mathbb{N}$, then their {\em join} $\sigma\ast \tau$ is defined as $\{\lambda p + \mu q \, | \, p\in\sigma, q\in\tau; \, \lambda, \mu \in [0, 1]$ and $ \lambda + \mu = 1\}$. Similarly, we define the {\em join} of two skew simplicial complexes $K_1$ and $K_2$ as $\{\sigma \ast \tau | \sigma\in K_1, \tau\in K_2\}$. The {\em link} of a face $\sigma$ in a simplicial complex $K$ is defined as $\{\tau\in K \,|\, \sigma\cap\tau=\emptyset, \tau\ast\sigma\in K\}$. The link of $\sigma$ is denoted by $\lk(\sigma,K)$. The {\em star} of a face $\sigma$ in a simplicial complex $K$ is defined as $\{\beta | \beta\leq \sigma\ast\alpha$ and $\alpha\in \lk (\sigma, K)\}$, and is denoted by $\Star (\sigma, K)$.

A simplicial complex $K$ is called {\em pure} if all its facets possess the same dimension. A simplicial complex $K$ is called a {\em normal $d$-pseudomanifold without a boundary} if it is pure, every face of dimension $d-1$ is contained in precisely two facets, and the link of every face of dimension $\leq d-2$ is connected. If some faces of dimension $d-1$ are contained in only one facet, then $K$ is called a {\em normal $d$-pseudomanifold with boundary}.
For a normal $3$-pseudomanifold $K$, the link of a vertex $v$ is a triangulated closed connected surface. If $|\lk(v)|\cong \mathbb{S}^2$, then we call $v$ a non-singular vertex; otherwise, $v$ is called a singular vertex. Let $K$ be a simplicial complex. Let $V(K)$, $E(K)$, $F(K)$, and $T(K)$ represent the counts of vertices, edges, faces, and tetrahedra in $K$, respectively. In this article, we will use the abbreviations $V$, $E$, $F$, and $T$ to refer to $V(K)$, $E(K)$, $F(K)$, and $T(K)$ when the context is clear that we are discussing simplicial complex $K$. The numbers $g_2$ and $g_3$ for a $3$-dimensional simplicial complex $K$ are defined as $g_2(K)=E-4V+10,$ and $g_3(K)= F-3E+6V-10.$
 \begin{remark}\label{7vertices}
 	{\rm Let $\{[1,2,4]$, $ [1,3,4]$, $ [1,2,5]$, $ [1,3,6]$, $ [1,5,6]$, $ [2,4,6]$, $ [2,3,5]$, $ [2,3,6]$, $ [3,4,5],$ $[4,5,6]\}$ denote a set of facets comprising the triangulation $L$ of $\mathbb{RP}^2$ with six vertices.
Consider $K=(uv\ast\lk(6))\cup(u\ast L\setminus \{6\})\cup(v\ast L\setminus \{6\})$. We define $K$ as the one-vertex suspension of the complex $L$. Topologically, $|K|$ represents the suspension of $|L|$. Consequently, $\{[1,2,4,u]$, $[1,3,4,u]$, $[1,2,5,u]$, $[1,2,4,v]$, $[1,3,4,v]$, $[1,2,5,v]$, $[1,3,u,v]$, $[1,5,u,v]$, $[2,4,u,v]$, $[2,3,5,u]$, $[2,3,5,v]$, $[2,3,u,v]$, $[3,4,5,u]$, $[3,4,5,v]$, $[4,5,u,v]\}$ forms the set of facets of $K$.}
\end{remark}
  
\begin{proposition}[Lemma $2.6$, \cite{BSR1}]\label{proposition:lower bound}
	Let $K$ be a normal $3$-pseudomanifold, and $v\in K$ be a vertex in $K$. Then $g_2(K)\geq g_2(\Star (v,K)) = g_2(\lk (v,K))$.
\end{proposition}
For definitions of operations like connected sum, bistellar $1$-move, edge contraction, edge expansion, vertex folding, and edge folding, please see \cite{BSR1}.
The subsequent results can be obtained by substituting $n=1$ into \cite[Theorem 4.4]{BSR1} and $m=1$ into \cite[Theorem 5.5]{BSR1}. 
\begin{proposition}\label{proposition:1-singularity}
	Let $K$ be a normal $3$-pseudomanifold with exactly one singularity at $t$, such that $|\lk (t,K)|$ is a torus or Klein bottle. Then, $g_2(K) \leq 15$ implies that $K$ is obtained from some boundary complexes of $4$-simplices by a sequence of operations of types connected sums, bistellar $1$-moves, edge contractions, edge expansions, and a vertex folding. Furthermore, $|K|$ is a handlebody with its boundary coned off.
\end{proposition}
\begin{proposition}\label{proposition:2-singularity}
	Let $K$ be a normal $3$-pseudomanifold with exactly two $\mathbb{RP}^2$ singularities. Then, $g_2{(K)} \leq 12$ implies $K$ is obtained from some boundary complexes of $4$-simplices by a sequence of operations of types connected sums, bistellar $1$-moves, edge contractions, edge expansions, and an edge folding. 
\end{proposition}
\begin{remark}\label{suspension}{\rm If $K$ is a normal $3$-pseudomanifold and $K^\psi_{uv}$ is obtained from $K$ by an edge folding at $uv$, then $g_2(K^\psi_{uv}) = g_2(K)+ 3.$
Furthermore, if $K$ is a triangulation of a 3-sphere then $K^\psi_{uv}$ is a triangulation of the suspension of $\mathbb{RP}^2$ because $|K^\psi_{uv}|$ is obtained from a 3-ball by the following procedure: $(i)$ choose two connected regions in the boundary of the 3-ball having a point in common, $(ii)$ identify those two regions in reverse orientation, and $(iii)$ then take the resulting topological space with its boundary conned off.
	}
\end{remark}
\section{Average edge order of a normal $3$-pseudomanifold}

Consider $K$ as a normal $3$-pseudomanifold with singularities, containing $n$ singular vertices denoted by ${t_1, t_2, \ldots, t_r, p_1, p_2, \ldots, p_{r'}}$, where $r+r'=n$ for some $r, r' \in \mathbb{N}$. According to surface classification, each vertex $x\in K$ has a link, $\lk(x)$, which is either a triangulated sphere, a triangulation of connected sums of tori, or a triangulation of connected sums of projective planes. Specifically, if $|\lk(t_i)|\cong \#{h_i}\mathbb{T}^2$ and $|\lk(p_j)|\cong\#{m_j}\mathbb{RP}^2$ for $h_i\geq 0$, $m_j\geq 0$, $1\leq i\leq r$, and $1\leq j\leq r'$, where $\mathbb{T}^2$ and $\mathbb{RP}^2$ represent a torus and a projective plane, respectively, then $\chi(t_i)=2-2h_i$ and $\chi(p_j)=2-m_j$.

Let $V, E, F$, and $T$ denote the count of vertices, edges, faces, and tetrahedra in $K$, respectively. Considering $g_2(K)=E-4V+10$ and $g_3(K)= F-3E+6V-10$, it follows that $g_2(K)+g_3(K)=F-2E+2V$.

\noindent Additionally, we have $$g_2(K)+g_3(K)=\sum_{v\in K}(2-\chi(\lk(v))).$$
\noindent Thus, the average edge order of a normal $3$-pseudomanifold is, 

\begin{equation*}
\begin{split} 
\mu_0(K)&=\frac{3F}{E}=3\frac{(g_2(K)+g_3(K)+2E-2V)}{E}\\
&=3\frac{(\sum_{v\in K}(2-\chi(\lk(v)))+2E-2V)}{E}\\
&=3\frac{(2V-\sum_{v\in K}\chi(\lk(v))+2E-2V)}{E}\\
&=3\frac{(-\sum_{v\in K}\chi(\lk(v))+2E)}{E}\\
&=6-3\frac{\sum_{v\in K}\chi(\lk(v))}{E}\\
&=6-3\frac{2(V-n)}{E}-3\frac{\sum_{i=1}^{r}\chi(\lk(t_i))}{E}-3\frac{\sum_{i=1}^{r'}\chi(\lk(p_i))}{E}\\
&=6-6\frac{(V-n)}{E}-3\frac{\sum_{i=1}^{r}(2-2h_i)}{E}-3\frac{\sum_{i=1}^{r'}(2-m_i)}{E}\\
&=6-6\frac{V}{E}+6\frac{n}{E}-6\frac{(r+r')}{E}+6\frac{\sum_{i=1}^{r}h_i}{E}+3\frac{\sum_{i=1}^{r'}m_i}{E}\\
&=6-6\frac{V}{E}+6\frac{\sum_{i=1}^{r}h_i}{E}+3\frac{\sum_{i=1}^{r'}m_i}{E}.\\
\end{split}
\end{equation*}

\noindent {\em Proof of Theorem} \ref{theorem:maintheorem} $(ii)$:	
		
\noindent Suppose that $\mu_0(K)\leq\frac{9}{2}$, then we obtain

\begin{equation*}
\begin{split} 
	6-6\frac{V}{E}+6\frac{\sum_{i=1}^{r}h_i}{E}+3\frac{\sum_{i=1}^{r'}m_i}{E}&\leq \frac{9}{2}\\
	\implies -6\frac{V}{E}+6\frac{\sum_{i=1}^{r}h_i}{E}+3\frac{\sum_{i=1}^{r'}m_i}{E}&\leq\frac{9}{2}-6\\
		\implies 4\frac{V}{E}-4\frac{\sum_{i=1}^{r}h_i}{E}-2\frac{\sum_{i=1}^{r'}m_i}{E}&\geq 1\\
		\implies E-4V+10\quad &\leq 10-4\sum_{i=1}^{r}h_i-2\sum_{i=1}^{r}m_i\\
		\implies  g_2(K) \quad &\leq 10-4\sum_{i=1}^{r}h_i-2\sum_{i=1}^{r}m_i.\\
\end{split}
\end{equation*}

\medskip	
	Consider $h$ as the maximum value in the set $\{h_1,h_2,...,h_r\}$, and let $m$ be the maximum value in $\{m_1,m_2,...,m_r'\}$. Given Lemma \ref{proposition:lower bound}, which states that $g_2(K)\geq g_2(\lk(t))$ for any $t\in K$, we derive:  
	$$6h\leq  10-4\sum_{i=1}^{r}h_i-2\sum_{i=1}^{r}m_i$$ and $$3m\leq  10-4\sum_{i=1}^{r}h_i-2\sum_{i=1}^{r}m_i.$$ 

\noindent	It is evident that the inequalities above are satisfied if and only if one of the following situations arises:

\noindent	$(a)$ $n=1$, $h=1$, $m=0$ 

\noindent	$(b)$ $n=1$, $h=0$, $m=1$
	
\noindent	$(c)$ $n=1$, $h=0$, $m=2$ 
	
\noindent	$(d)$ $n=2$, $h=0$, $m=1$

\noindent	$(e)$ $n=3$, $h=0$, $m=1.$
\medskip\\		         
\noindent 
Due to the fact that the sum of Euler characteristics for the vertex links of a normal $3$-pseudomanifold is always even, we can infer that cases $(b)$ and $(e)$ are not feasible.
\medskip\\
\noindent\textbf{Cases $(a), (c)$:}
In both cases, $g_2(K)\leq 6$, and $K$ features precisely one singularity at either $t_1$ or $p_1$, where $\lk(t_1)$ forms a triangulated torus and $\lk(p_1)$ constitutes a triangulated Klein bottle. Moreover, $g_2(\lk(t_1))=g_2(\lk(p_1))=6$. Additionally, according to Proposition \ref{proposition:lower bound}, we establish $g_2(K)\geq 6$. Consequently, $g_2(K) = 6$. Henceforth, according to Proposition \ref{proposition:1-singularity}, $K$ is obtained from some boundary complexes of $4$-simplices by a sequence of operations, including connected sums,  bistellar $1$-moves, edge contractions, edge expansions, and vertex folding. Furthermore, $|K|$ is a handlebody with its boundary coned off. 	

\noindent \textbf{Case $(d)$:} In this instance, $K$ encompasses two $\mathbb{RP}^2$ singularities at vertices $p_1$ and $p_2$. According to Proposition \ref{proposition:lower bound}, we have $3\leq g_2(K)\leq 6$. Consequently, by Proposition \ref{proposition:2-singularity}, we deduce that $K$ is obtained from some boundary complexes of $4$-simplices by a sequence of possible operations, including connected sums,  bistellar $1$-moves, edge contractions, edge expansions, and an edge folding. Moreover, $K$ is the triangulation of suspension of $\mathbb{RP}^2$ by Remark \ref{suspension}. 

Hence, by combining cases $(a), (c)$ and $(d)$, proof of Theorem \ref{theorem:maintheorem} $(ii)$ follows.\\

\noindent{\em Proof of Theorem} \ref{theorem:maintheorem} $(i)$:
 
 \noindent Consider $K$ as a normal $3$-pseudomanifold with singularities, where $\mu_0(K)\leq \frac{30}{7}$. Therefore, $\mu_0(K)\leq \frac{30}{7}< \frac{9}{2}$. Consequently, based on the preceding discussions, we have only three potential cases for $K$: $(a),$ $(c),$ and $(d)$. In both cases $(a)$ and $(c)$, $\mu_0(K)=\frac{9}{2}>\frac{30}{7}$, rendering them implausible. Now, let's discuss the case $(d)$. Assume $K$ contains two $\mathbb{RP}^2$ singularities at vertices $u$ and $v$. Hence, $3\leq g_2(K)\leq 6$. Then, we obtain

 \begin{equation*}
     \begin{split}
         \mu_0(K)=6-\frac{6(V-1)}{E}&\leq\frac{30}{7}\\
 \implies 6-\frac{6(V-1)}{4V-10+g_2(K)}&\leq\frac{30}{7}\\
\implies \frac{18V+6(g_2(K-9))}{4V-10+g_2(K)}&\leq\frac{30}{7}\\
\implies 126V-378+42g_2(K)&\leq 120V+30g_2(K)-300\\
 \implies V \quad&\leq 13-2g_2(K).\\
     \end{split}
 \end{equation*}
\noindent 
Considering that for any normal $3$-pseudomanifold $\Delta$, $V(\Delta)\geq 7$, the aforementioned inequality holds if and only if $g_2(K)=3$ and $V(K)=7$. Consequently, this implies that $\mu_0(K)=\frac{30}{7}$, and $K$ represents a triangulation of the one-vertex suspension of $\mathbb{RP}^2$. Conversely, according to Remark \ref{7vertices}, we obtain a triangulation $K$ of the one-vertex suspension of $\mathbb{RP}^2$ with seven vertices, satisfying $\mu_0(K)=\frac{30}{7}$ and $g_2(K)=3$.
This concludes the proof of Theorem \ref{theorem:maintheorem} $(i)$.\\

\noindent {\em Proof of Theorem} \ref{theorem:maintheorem} $(iii)$:

\noindent 
Let $h$ and $m$ denote the same as previously stated. Therefore, $g_2(K)\geq 6h$, and $g_2(K)\geq 3m$. Consequently, $E(K)\geq 6h$ and $E(K)\geq 3m$.

\noindent Additionally, we have
		$$\mu_0(K)=6-6\frac{V}{E}+6\frac{\sum_{i=1}^{r}h_i}{E}+3\frac{\sum_{i=1}^{r'}m_i}{E}.$$ 
	\noindent Thus,	
\begin{equation*}
    \begin{split}
        \mu_0(K)&=6-6\frac{V}{E}+6\frac{\sum_{i=1}^{r}h_i}{E}+3\frac{\sum_{i=1}^{r'}m_i}{E}\\
        &\leq 6+6\frac{rh}{E}+3\frac{r'm}{E}\\
        &<6+6\frac{rh}{6h}+3\frac{r'm}{3m}\\
     &=6+r+r'\\
     &=6+n.\\
    \end{split}
\end{equation*}
Hence, $\mu_0(K)<6+n$. 
\medskip\\	
Consider $V_m, E_m,$ and $F_m$ as the number of vertices, edges, and faces, respectively, of the minimal triangulation $\mathbb{T}_m$ representing a torus with genus $m\geq 3$. According to \cite{JungermanRingel}, we find:
$V_m=\ceil{\frac{7+\sqrt{1+48m}}{2}}$, $E_m=3V_m+6m-6$, $F_m=2V_m+4m-4$. Now, if we take the suspension of the surface $\mathbb{T}_m$, we obtain a normal $3$-pseudomanifold denoted as $K_m$ with two singularities, where:
$V(K_m)=V_m+2$, $E=2V_m+E_m$, and $F=2E_m+F_m$.
\noindent	Then, $$\mu_0(K_{m})=\frac{3F}{E}=\frac{6E_m+3F_m}{2V_m+E_m}
=\frac{24V_m+48m-48}{5V_m+6m-6}.$$
Thus,  $\mu_0(K_{m})\rightarrow 8$ as $\ m\rightarrow\infty$.
This completes the proof of Theorem \ref{theorem:maintheorem}$(iii)$.
 \begin{remark}
 	{\rm 
If $K_1$ represents a normal $3$-pseudomanifold obtained from a stacked sphere $K'$ with $g_2(K')=0$ via a vertex folding, then the following relations hold: $E(K')=4V(K')-10$, $F(K')=6V(K')-20$, $T(K')=3V(K')-10$. Moreover, $V(K_1)=V(K')-3$, $E(K_1)=E(K')-6=4V(K')-16$, $F(K_1)=F(K')-4=6V(K')-24$, $T(K_1)=T(K')-2=3V(K')-12$. Now, we find that $$\mu_0(K_1)=\frac{3F}{E}=3\frac{6V-24}{4V-16}=\frac{9}{2}(\frac{V-4}{V-4})=\frac{9}{2}.$$} 
 \end{remark} 
 \begin{corollary}
	If $K$ is a normal $3$-pseudomanifold and $\mu_0(K)<\frac{30}{7}$, then $K$ is triangulation of a $3$-sphere.
	\end{corollary}
 \begin{proof}
    If $K$ represents a normal $3$-pseudomanifold and $\mu_0(K)<\frac{30}{7}$, according to Theorem \ref{theorem:maintheorem}, $K$ does not contain any singular vertices. This indicates that $K$ forms a triangulation of a $3$-manifold, where $\mu_0(K)<\frac{30}{7}$. Consequently, the result follows from Proposition \ref{LuoStong}$(iii)$.
 \end{proof}
\section{On the upper bound of average edge order}
The upper bound provided in Theorem \ref{theorem:maintheorem} varies according to the number of singularities. In this context, we present several examples of normal $3$-pseudomanifolds with singularities that possess a higher average edge order, sourced from the simplicial complex library of GAP \cite{jspreer}. Denote by [SCLib, $n$] the simplicial complex located at the $n$-th position within the simplicial complex library in GAP. 

\noindent \textbf{Examples:}
\begin{enumerate}
	\item Let $K_1=[\text{SCLib}, 60]$ and $K_2=[\text{SCLib}, 61]$, which represent two normal $3$-pseudomanifolds with the vector $(V,E,F,T)=(11, 55, 154, 77)$. Notably, both $K_1$ and $K_2$ exhibit $\mu_0(K_1)=\mu_0(K_2)=8.4$. In both complexes, the link of each vertex forms a triangulation consisting of $6$ copies of $\mathbb{RP}^2$. The automorphism groups for $K_1$ and $K_2$ are $D_{22}$ (a dihedral group with $22$ elements) and $C_{11}$ (a cyclic group with $11$ elements), respectively.
	
	\item Let $K_3=[\text{SCLib}, 540]$ be a normal $3$-pseudomanifold with the vector $(V,E,F,T)$ $=(17, 136, 544, 272)$. Consequently, $\mu_0(K_3)=12$. The link of each vertex in $K_3$ forms a triangulation consisting of $18$ copies of $\mathbb{RP}^2$. The automorphism group of $K_3$ is $C_{17}:C_{16}$, a semidirect product of cyclic groups of orders $16$ and $17$.
	
	\item Let $K_5=[\text{SCLib}, 587]$ be a normal $3$-pseudomanifold with the vector $(V,E,F,T)$ $=(19, 171, 684, 342)$. Consequently, $\mu_0(K_5)=12$. The link of each vertex in $K_5$ forms a triangulation consisting of $20$ copies of $\mathbb{RP}^2$. The automorphism group of $K_5$ is $C_{19}:C_{18}$.
	
	\item Let $K_4=[\text{SCLib}, 541]$ be a normal $3$-pseudomanifold with the vector $(V,E,F,T)$ $=(17, 136, 680, 340)$. Consequently, $\mu_0(K_4)=15$. The link of each vertex in $K_4$ forms a triangulation consisting of $26$ copies of $\mathbb{RP}^2$. The automorphism group of $K_4$ is the same as the automorphism group of $K_3$, i.e., $C_{17}:C_{16}$.

\end{enumerate}

It is interesting to note that in all the above examples, each vertex in the complexes is a singular vertex. This draws attention towards the study of normal $3$-pseudomanifolds with all vertices being singular. Additionally, it's worth mentioning that all of these complexes exhibit neighborly properties, with even the example of $K_4$ being $3$-neighborly. If a normal $3$-pseudomanifold $K$ is $3$-neighborly, then $\mu_0(K)=\frac{3\binom{V}{3}}{\binom{V}{2}}=V-2$. 
These observations motivate us to study the following questions of interest.
\begin{ques}
    How many vertices are needed at minimum to construct a neighborly normal $3$-pseudomanifold, where the link of each vertex forms a triangulated surface of genus $g$ for $g\in\mathbb{N}$? 
 \end{ques}
\begin{ques}\label{aeo}
Is there a $3$-neighborly normal $3$-pseudomanifold $K$ with $V(K)>n$ for every positive integer $n\in\mathbb{N}$?  
 \end{ques}

\noindent {\bf Acknowledgement:} 
We would like to thank Prof. Wolfgang Kühnel for his comments and suggestions on the manuscript. The first author is supported by the Mathematical Research Impact
Centric Support Research Grant (MTR/2022/000036) by ARNF, India. The second author is supported by CSIR (India).

\smallskip


\smallskip

\noindent {\bf Declarations}

\noindent {\bf Conflict of interest:} No potential conflict of interest was reported by the authors.
	
	{

\end{document}